\documentclass{amsart}
\title[Flat bundles and Hyper-Hodge decomposition  on  solvmanifolds]
{Flat bundles and Hyper-Hodge decomposition  on  solvmanifolds}

\author{Hisashi Kasuya}

\usepackage{amssymb}
\usepackage{amsmath}
\usepackage{amscd}
\usepackage{amstext}
\usepackage{amsfonts}
\usepackage[all]{xy}

\theoremstyle{plain}

\theoremstyle{plain}

\theoremstyle{plain}

\theoremstyle{plain}
\newtheorem{theorem}{Theorem}[section] 
\theoremstyle{remark}
\newtheorem{remark}{Remark}
\theoremstyle{Main result}
\newtheorem{main result}{Main result}
\theoremstyle{lemma}
\newtheorem{lemma}[theorem]{Lemma}
\theoremstyle{definition}
\newtheorem{definition}[theorem]{Definition}
\theoremstyle{proposition}
\newtheorem{proposition}[theorem]{Proposition}
\theoremstyle{corollary}
\newtheorem{corollary}[theorem]{Corollary}
\theoremstyle{remark}

\theoremstyle{remark}
\newtheorem{problem}{Problem}
\theoremstyle{remark}

\theoremstyle{assumption}
\newtheorem{assumption}[theorem]{Assumption}

\address[Hisashi Kasuya]{Department of Mathematics, Tokyo Institute of Technology, 2-12-1 Ookayama, Meguro-ku, Tokyo 152-8551, JAPAN}
\email{kasuya@math.titech.ac.jp}


\newcommand{\C}{\mathbb{C}}
\newcommand{\R}{\mathbb{R}}

\newcommand{\Z}{\mathbb{Z}}
\newcommand{\g}{\frak{g}}
\newcommand{\n}{\frak{n}}

\begin{document} 

\maketitle
\begin{abstract}
We study   rank $1$  flat bundles over solvmanifolds  whose cohomologies are non-trivial.
By using Hodge theoretical properties  for all topologically trivial rank $1$ flat bundles, we represent the structure theorem of K\"ahler solvmanifolds as extensions of Hasegawa's result  and Benson-Gordon's result for nilmanifolds.
\end{abstract}
\section{Introduction}

Let $M$ be a compact manifold and $(A^{\ast}_{\C}(M),d)$ (resp. $(A^{\ast}_{\R}(M),d)$) the $\C$-valued (resp. $\R$-valued) de Rham complex and $A^{\ast}_{\C\,  cl}(M)$ (resp. $A^{\ast}_{\R\,  cl}(M)$) the subspace of the $\C$-valued (resp. $\R$-valued) closed forms.
We denote by $F(M) $ the set of   isomorphism classes of flat $\C$-line bundles $E_{\phi}=(M\times \C, d+\phi)$ where $M\times \C$ is a topologically trivial line bundle and $\phi\in A^{1}_{\C\, cl}(M)$.
We denote by ${\mathcal C}(\pi_{1}(M))$ the space of characters of $\pi_{1}(M)$ which can be factored as
\[\pi_{1}(M)\to H_{1}(\pi_{1}(M),\Z)/({\rm torsion})\to \C^{\ast}.
\]
Then we have the $1-1$ correspondence $\iota: F(M)\to {\mathcal C}(\pi_{1}(M))$ such that $\iota(E_{\phi})(\gamma)=e^{\int_{\gamma}\phi}$ for $\gamma\in \pi_{1}(M)$.
Because of  this, the map $A^{\ast}_{\C\, cl}(M)\ni \phi\mapsto E_{\phi}\in F(M)$ induces a surjection $H^{\ast}(M,\C)\ni [\phi]\mapsto E_{\phi}\in F(M)$.
We consider the $\R^{\ast}$-action on $F(M)$ such that 
\[\mu_{t}(E_{\phi})=E_{t{\rm  Re}\phi+\sqrt{-1}{\rm Im} \phi}
\]
for $t\in \R^{\ast}$.
By the definition,  a unitary flat bundle $E_{\phi}\in F(M)$ (i.e. $\iota(E_{\phi})$ is a unitary character) is fixed by the $\R^{\ast}$-action.
We can consider  the cochain complex $(A^{\ast}_{\C}(M),d+\phi)$ as
the de Rham complex with values in a flat bundle $E_{\phi}$.
We denote by $H^{\ast}(M, d+\phi)$ the cohomology of $(A^{\ast}_{\C}(M),d+\phi)$.
We define ${\mathcal J}^{p}(M)=\{ E_{\phi}\in F(M)\vert  H^{p}(M,d+\phi)\not=0\}$ and denote ${\mathcal J}(M)=\bigcup{\mathcal J}^{p}(M)$.

The main objects of this paper are solvmanifolds.
Solvmanifolds are compact homogeneous spaces $G/\Gamma$ of simply connected solvable Lie groups $G$ by lattices (i.e. cocompact discrete subgroups) $\Gamma$.
For solvmanifolds $G/\Gamma$, in this paper,   we describe the set ${\mathcal J}^{p}(G/\Gamma)$ for each $p$ by using the adjoint representations of $G$ (see Section \ref{sesollo}).

\begin{remark}\label{remjul}
Let ${\breve{\mathcal J}}^{p}(M)$ denote the set of  isomorphism classes of (not necessarily topologically trivial) flat $\C$-line bundles over a compact manifold $M$ whose cohomologies are non-trivial.
The set  ${\breve {\mathcal J}}^{p}(M)$ is called the cohomology jump locus of $M$.
The set ${\breve {\mathcal J}}^{p}(M)$  was studied   by several authors (for example \cite{Ara}, \cite{Sim2}, \cite{KP}, \cite{MPa} and \cite{DPS}).
For solvmanifolds $G/\Gamma$, we have ${\mathcal J}^{p}(G/\Gamma)={\breve {\mathcal J}}^{p}(G/\Gamma)$ (see Corollary \ref{soto}).
\end{remark}

\begin{definition}
 $M$ has the $\mu_{\R^{\ast}}$-symmetry on cohomologies if for each $E_{\phi}\in F(M)$ and $t\in \R^{\ast}$, we have 
\[\dim H^{r}(M,d+\phi)=\dim H^{r}(M,d+t{\rm  Re}\phi+\sqrt{-1}{\rm Im} \phi)\] for each $r$.
\end{definition}
If $M$ has the $\mu_{\R^{\ast}}$-symmetry on cohomologies and there exists a non-unitary flat bundle $E_{\phi}$ such that $H^{\ast}(M, d+\phi)\not=0$, then ${\mathcal J}(M)$ is an infinite set.
Let
\[\overline {A}^{\ast}(M)=\bigoplus_{E_{\phi}\in F(M)}\left(A^{\ast}_{\C}(M), d+\phi\right).
\]
Then by the isomorphism $E_{\phi}\otimes E_{\varphi}\cong E_{\phi+\varphi}$,  $\overline {A}^{\ast}(M)$ is a differential graded algebra.
\begin{definition}
 $M$ is hyper-formal if the differential graded algebra $\overline {A}^{\ast}(M)$ is formal in the sense of Sullivan (\cite{Sul}).
\end{definition}

\begin{definition}
We suppose that $M$ admits a symplectic form $\omega$.
 $(M,\omega)$ is hyper-hard-Lefschetz if for each $E_{\phi}\in F(M)$  the linear map  \[[\omega]^{n-i}\wedge: H^{i}(M,d+\phi)\to H^{2n-i}(M,d+\phi)
\]
is an isomorphism for any $i\le n$ where $\dim M=2n$.

\end{definition}

Let $(M,J)$ be  a compact complex manifold.
Consider the double complex $(A^{\ast,\ast}(M),\partial, \bar\partial)$ and the Dolbeault cohomology $H^{\ast,\ast}(M, \bar\partial)$.
We also consider the Bott-Chern cohomology $H^{\ast,\ast}(M,\partial\bar\partial)$ defined by
\[H^{\ast,\ast}(M,\partial\bar\partial)=\frac{{\rm ker}\,\partial\cap {\rm ker}\,\bar\partial}{{\rm im}\,\partial\bar\partial}.
\]
We say that  $(M,J)$ admits the strong-Hodge-decomposition if the natural maps
\[H^{\ast,\ast}(M,\partial\bar\partial)\to H^{\ast,\ast}(M,\bar\partial),\,\,\,\,\,\,\,\, {\rm Tot}^{\ast}H^{\ast,\ast}(M,\partial\bar\partial)\to H^{\ast}(M,d)
\]
are isomorphisms (see \cite{dem}).
This condition is equivalent to the $\partial\bar\partial$-lemma as in \cite{DGMS} (see \cite{AT}) and hence this condition implies the formality in the sense of Sullivan. 
Let $E_{\phi}\in F(M)$ be a unitary flat bundle.
Then we have $[\phi]\in \sqrt{-1}H^{\ast}(M,\R)$ and hence we can take $\phi\in \sqrt{-1}A^{\ast}_{\R\,  cl}(M)$.
We consider the decomposition $d+\phi=(\partial+\phi^{1,0})+(\bar\partial+\phi^{0,1})$ where $\phi^{1,0}$ and $\phi^{0,1}$ are $(1,0)$ and $(0,1)$ components of $\phi$ respectively.
Then $(A^{\ast}_{\C}(M),\bar\partial+\phi^{0,1})$ is considered as the Dolbeault complex with values in a flat holomorphic bundle $E_{\phi}$.
For a holomorphic $1$-form $\theta\in H^{1,0}(M,\bar\partial)$, we consider the differential operator $\bar\partial+\phi^{0,1}+\theta$ on $A^{\ast}_{\C}(M)$.
We denote by $H^{\ast}(M, \bar\partial+\phi^{0,1}+\theta)$ the cohomology of $(A^{\ast}_{\C}(M), \bar\partial+\phi^{0,1}+\theta)$.
We assume that $(M,J)$  admits the strong-Hodge-decomposition.
Then we have $H^{1}(M,d)\cong{\rm ker}\partial
\cap {\rm ker}\bar\partial=H^{1,0}(M,\bar\partial)\oplus \overline{H^{1,0}(M,\bar\partial)}$.
Because of this, for  holomorphic $1$-forms  $\theta$ and $\vartheta$, we have the flat bundle $E_{\vartheta-\bar\vartheta+\theta+\bar\theta}\in F(M)$ and the decomposition
\[d+\vartheta-\bar\vartheta+\theta+\bar\theta=(\partial+\vartheta+\bar\theta)+(\bar\partial-\bar\vartheta+\theta)
\]
of differential operators.
Moreover, each element in  $ F(M)$ can be written as $E_{\vartheta-\bar\vartheta+\theta+\bar\theta}$.
We denote by 
\[H^{\ast}\left(M, (\partial+\vartheta+\bar\theta)(\bar\partial-\bar\vartheta+\theta)\right)\] 
the vector space
\[\frac{{\rm ker}(\partial+\vartheta+\bar\theta)\cap{\rm ker}(\bar\partial-\bar\vartheta+\theta)}{{\rm im}  (\partial+\vartheta+\bar\theta)(\bar\partial-\bar\vartheta+\theta)}.
\]
\begin{definition}
Let $(M,J)$ be a compact complex manifold  admitting the strong-Hodge-decomposition.
\begin{itemize}
\item $(M,J)$ satisfies the hyper-strong-Hodge-decomposition if  for each holomorphic $1$-forms  $\theta$ and $\vartheta$ 
the natural maps
\[H^{\ast}\left(M, (\partial+\vartheta+\bar\theta)(\bar\partial-\bar\vartheta+\theta)\right)\to H^{\ast}\left(M, \bar\partial-\bar\vartheta+\theta\right), \]
 \[H^{\ast}\left(M, (\partial+\vartheta+\bar\theta)(\bar\partial-\bar\vartheta+\theta)\right)\to H^{\ast}\left(M, \partial+\vartheta+\bar\theta\right) \]  
and 
\[H^{\ast}\left(M, (\partial+\vartheta+\bar\theta)(\bar\partial-\bar\vartheta+\theta)\right)\to H^{\ast}\left(M, d+\vartheta-\bar\vartheta+\theta+\bar\theta\right) \] 
are isomorphisms.
\end{itemize}
\end{definition}
We have the following relations similar to \cite{DGMS}.

\begin{proposition}
Let $(M,J)$ be a compact complex manifold admitting the strong-Hodge-decomposition.
Then if $(M,J)$ satisfies the hyper-strong-Hodge-decomposition,   $M$ has the $\mu_{\R^{\ast}}$-symmetry on cohomologies and $M$ is hyper-formal.
\end{proposition}

In  case $(M,J)$ admits a K\"ahler structure, for  holomorphic $1$-forms  $\theta$ and $\vartheta$,  the pair $(E_{\vartheta-\bar\vartheta}, \theta)$ is considered as a Higgs bundle in the sense of Simpson \cite{Sim}.
By using the harmonic metric on a Higgs bundle, we can show the K\"ahler identity (see \cite[Section 2]{Sim}).
Similarly to the proof of the  ordinary strong-Hodge-decomposition (see \cite{dem}), we have the following theorem.
\begin{theorem}{\rm (\cite[Section 2]{Sim})}\label{simps}
Let $(M,J,\omega)$ be a compact K\"ahler manifold.
Then $(M,J)$ satisfies the hyper-strong-Hodge-decomposition and  $(M,\omega)$ is hyper-hard-Lefschetz. 
\end{theorem}

In this paper, we study the above properties on solvmanifolds.

\begin{theorem}\label{mainte}
Let $M$ be a $2n$-dimensional  solvmanifold.
Then the following conditions are equivalent

\begin{enumerate}
\item $M$ admits a complex structure $J$ and $(M,J)$ admits the strong-Hodge-decomposition and the  hyper-strong-Hodge-decomposition.
\item $\dim H^{1}(M,\R)$ is even,   $M$ has the $\mu_{\R^{\ast}}$-symmetry on cohomologies and $M$ is hyper-formal.
\item$M$ has the $\mu_{\R^{\ast}}$-symmetry on cohomologies,  $M$ admits a symplectic form $\omega$ and
 $(M,\omega)$ is hyper-hard-Lefschetz.
\item $M$ is written as $G/\Gamma$ where $G=\R^{2k}\ltimes_{\varphi} \R^{2l} $ such that the action $\varphi:\R^{2k}\to {\rm Aut}(\R^{2l})$ is  semi-simple and 
for any $x\in \R^{2k}$  all the eigenvalues of $\varphi(x)$ are unitary.
\item $M$ admits a K\"ahler structure.
\end{enumerate}
\end{theorem}

\begin{remark}
In \cite{H}, Hasegawa showed that formal nilmanifolds are tori and in particular nilmanifolds admitting the strong-Hodge-decomposition are tori.
In \cite{BG}, Benson-Gordon showed that hard-Lefschetz symplectic nilmanifolds are tori.
These results give the structure theorem for K\"ahler nilmanifolds.
Theorem \ref{mainte} can be regarded as extensions of Hasegawa's result  and Benson-Gordon's result for nilmanifolds.
\end{remark}

\begin{remark}
The equivalence of (4) and (5) in Theorem \ref{mainte} were already proved by Hasegawa in \cite{H} by using Arapura-Nori's results in \cite{AN}.
But it is not  clear one can consider the result in \cite{H} as extensions of Hasegawa’s, Benson-Gordon’s result for nilmanifolds.
\end{remark}

\begin{remark}
Arapura-Nori's results in \cite{AN} follows from Arapura's earlier work in \cite{Ara}.
The proof of Theorem \ref{mainte} is similar to the Arapura's idea in \cite{Ara}.
But the proof of Theorem \ref{mainte} is independent of \cite{Ara}. 
\end{remark}

\begin{remark}
By Theorem \ref{mainte}, we have no example of a non-K\"ahler  solvmanifold satisfying the hyper-strong-Hodge-decomposition.
But there exist examples of non-K\"ahler  solvmanifolds satisfying the strong-Hodge-decomposition, the hyper-formality and the hyper-hard-Lefschetz property (see Section \ref{secex}).

\end{remark}
We suggest a new problem for non-K\"ahler geometry.
\begin{problem}

Provide  non-K\"ahler manifolds $M$ such that
%\item   $M$ has $\mu_{\R^{\ast}}$-symmetry  on cohomologies and non-trivial cohomology with values in a non-unitary local system.
%\item $M$ is hyper-formal.
%\item $M$ admits a symplectic form $\omega$ and
% $(M,\omega)$ is hyper-hard-Lefschetz.
%\item
  $M$ admit  complex structures $J$ and $(M,J)$ satisfy the strong-Hodge-decomposition and the hyper-strong-Hodge-decomposition.

\end{problem}

\section{Hyper-strong-Hodge-decomposition and $\mu_{\R^{\ast}}$-symmetry on cohomologies}

Let $(M,J)$ be  a compact complex manifold.
Let $E_{\phi}\in F(M)$ be a unitary flat bundle with $\phi\in \sqrt{-1}A^{\ast}_{\R\,cl}(M)$.
For a holomorphic $1$-form $\theta\in H^{1,0}(M, \bar\partial)$, we consider the differential operator $\bar\partial+\phi^{0,1}+\theta$ on $A^{\ast}_{\C}(M)$ and the cohomology $H^{\ast}\left(M, \bar\partial+\phi^{0,1}+\theta\right)$.
\begin{lemma}\label{si}
For any $t\in \C^{\ast}$,
we have 
\[\dim H^{\ast}\left(M, \bar\partial+\phi^{0,1}+\theta\right)=\dim H^{\ast}\left(M, \bar\partial+\phi^{0,1}+t\theta\right).\]

\end{lemma}
\begin{proof}
Considering the bi-grading $A^{r}_{\C}(M)=A^{p,q}(M)$.
Then  the cohomology $H^{\ast}\left(M,\bar\partial+\phi^{0,1}+\theta\right)$ is the total cohomology of the double complex $\left(A^{\ast,\ast}(M),\bar\partial+\phi^{0,1}, \theta\right)$.
Consider the spectral sequence $E^{\ast,\ast}_{\ast}$ of the double complex $\left(A^{\ast,\ast}(M),\bar\partial+\phi^{0,1}, \theta\right)$.
Then, as in \cite{BT} or \cite{CFGU}, an isomorphism
\[E_{r}^{p,q}\cong \frac{X^{p,q}_{r}}{Y_{r}^{p,q}}
\]
holds
where for $r\ge 2$,
\begin{multline*}X_{r}^{p,q}=\{\psi_{p,q}\in A^{p,q}(M)\vert  \bar\partial\psi_{p,q}+\phi^{0,1}\wedge\psi_{p,q}=0,\,\, \\
\exists \psi_{p+i,q+i}\in  A^{p+i,q-i}(M),\,\, \\
s.t. \,\,\theta \wedge\psi_{p+i-1,q-i+1}+\bar\partial\psi_{p+i,q-i}+\phi^{0,1}\wedge\psi_{p+i,q-i}=0 , \,\, 1\le i\le r-1\},
\end{multline*}
\begin{multline*}Y_{r}^{p,q}=\{\theta\wedge\omega_{p-1,q}+\bar\partial\omega_{p,q-1}+\phi^{0,1}\wedge\omega_{p,q-1}\in A^{p,q}(M)\vert \\ \exists \omega_{p-i,q+i-1}\in  A^{p-i,q+i-1}(M) \\
s.t. \,\,\theta\wedge \omega_{p-i,q+i-1}+\bar\partial\omega_{p-i+1,q+i-2}+\phi^{0,1}\wedge\omega_{p-i+1,q+i-2}=0 , \,\, 2\le i\le r-1\}.
\end{multline*}
For $t\in \C^{\ast}$, we also consider the spectral sequence $E^{\ast,\ast}_{\ast}(t)$ of the double complex $(A^{\ast,\ast}_{\C}(M),\bar\partial+\phi^{0,1}, t\theta)$.
Then we  have
\[E_{r}^{p,q}(t)\cong \frac{X^{p,q}_{r}(t)}{Y_{r}^{p,q}(t)}
\]
where for $r\ge 2$,
\begin{multline*}X_{r}^{p,q}(t)=\{\psi^t_{p,q}\in A^{p,q}(M)\vert  \bar\partial\psi^t_{p,q}+\phi^{0,1}\wedge\psi^t_{p,q}=0,\,\, \\
\exists \psi^t_{p+i,q+i}\in  A^{p+i,q-i}(M),\,\, \\
s.t. \,\,t\theta \wedge\psi^t_{p+i-1,q-i+1}+\bar\partial\psi^t_{p+i,q-i}+\phi^{0,1}\wedge\psi^t_{p+i,q-i}=0 , \,\, 1\le i\le r-1\},
\end{multline*}
\begin{multline*}Y_{r}^{p,q}(t)=\{t\theta\wedge\omega^t_{p-1,q}+\bar\partial\omega^t_{p,q-1}+\phi^{0,1}\wedge\omega_{p,q-1}\in A^{p,q}(M)\vert \\ \exists \omega^t_{p-i,q+i-1}\in  A^{p-i,q+i-1}(M) \\
s.t. \,\,t\theta\wedge \omega^t_{p-i,q+i-1}+\bar\partial\omega^t_{p-i+1,q+i-2}+\phi^{0,1}\wedge\omega^t_{p-i+1,q+i-2}=0 , \,\, 2\le i\le r-1\}.
\end{multline*}
For $\psi_{p,q}\in X_{r}^{p,q}$, considering $\psi_{p+i,q-i}^{t}=t^{i}\psi_{p+i,q-i}$, we can say $\psi_{p,q}\in X_{r}^{p,q}(t)$.
For $\theta\wedge\omega_{p-1,q}+\bar\partial\omega_{p,q-1}+\phi^{0,1}\wedge\omega_{p,q-1}\in Y_{r}^{p,q}$, considering $ \omega^t_{p-i,q+i-1}= t^{-i}\omega_{p-i,q+i-1}$ we can say \[\theta\wedge\omega_{p-1,q}+\bar\partial\omega_{p,q-1}+\phi^{0,1}\wedge\omega_{p,q-1}=t\theta\wedge\omega^{t}_{p-1,q}+\bar\partial\omega_{p,q-1}+\phi^{0,1}\wedge\omega_{p,q-1}\in Y_{r}^{p,q}.\]
By these relations, we have $X_{r}^{p,q}\cong X_{r}^{p,q}(t)$ and $Y_{r}^{p,q}\cong Y_{r}^{p,q}(t)$ and hence $E_{r}^{p,q}\cong E_{r}^{p,q}(t)$.
Hence, for sufficiently large $r$, we have
\[E_{\infty}=E_{r}\cong E_{r}^{p,q}(t)\cong E_{\infty}^{p,q}(t).
\]
Since the spectral sequences $E_{\ast}^{\ast,\ast}$ and $E_{\ast}^{\ast,\ast}(t)$ converge to the cohomologies $H^{\ast}\left(M, \bar\partial+\phi^{0,1}+\theta\right)$ and $ H^{\ast}\left(M,\bar\partial+\phi^{0,1}+t\theta\right) $ respectively,
 the lemma follows.
\end{proof}

\begin{proposition}\label{Hysym}
Let $(M,J)$ be a compact complex manifold admitting the strong-Hodge-decomposition.
Then if $(M,J)$ satisfies the hyper-strong-Hodge-decomposition,  $M$ has the $\mu_{\R^{\ast}}$-symmetry on cohomologies and $M$ is hyper-formal.
\end{proposition}

\begin{proof}
By the definition of the hyper-strong-Hodge-decomposition, for each  $1$-forms $\theta$ and $\vartheta$,  we have an isomorphism
\[H^{\ast}(M,d+\vartheta-\bar\vartheta+\theta+\bar\theta)\cong H^{\ast}(M,\bar\partial-\bar\vartheta+\theta)\cong H^{\ast}(M,\partial+\vartheta+\bar\theta) \]
and we can easily check that  the inclusion
\[\left({\rm ker}(\partial+\vartheta+\bar\theta), \bar\partial-\bar\vartheta+\theta\right)\to \left(A^{\ast}_{\C}(M),d+\vartheta-\bar\vartheta+\theta+\bar\theta\right)
\]
is a quasi-isomorphism.
We consider the quotient  map
\[\left({\rm ker}(\partial+\vartheta+\bar\theta), \bar\partial-\bar\vartheta+\theta\right)\to \left( H^{\ast}(M,\partial+\vartheta+\bar\theta), \bar\partial-\bar\vartheta+\theta \right).
\]
Then since the map  
 \[H^{\ast}\left(M, (\partial+\vartheta+\bar\theta)(\bar\partial-\bar\vartheta+\theta)\right)\to H^{\ast}\left(M, \partial+\vartheta+\bar\theta\right) \] 
is an isomorphism,
each cohomology class in $H^{\ast}\left(M, \partial+\vartheta+\bar\theta\right)$ is represented by a $(\bar\partial-\bar\vartheta+\theta)$-closed element and hence 
\[
\left( H^{\ast}(M,\partial+\vartheta+\bar\theta), \bar\partial-\bar\vartheta+\theta \right)\cong \left( H^{\ast}(M,\partial+\vartheta+\bar\theta), 0\right).\]
We can easily check that the map 
 \[\left({\rm ker}(\partial+\vartheta+\bar\theta), \bar\partial-\bar\vartheta+\theta\right)\to \left( H^{\ast}(M,d+\vartheta-\bar\vartheta+\theta+\bar\theta), 0 \right)
\]
is a quasi-isomorphism.
By the isomorphism $H^{\ast}(M,d+\vartheta-\bar\vartheta+\theta+\bar\theta)\cong H^{\ast}(M,\partial+\vartheta+\bar\theta)$,
we have the quasi-isomorphism
 \[\left({\rm ker}(\partial+\vartheta+\bar\theta), \bar\partial-\bar\vartheta+\theta\right)\to \left( H^{\ast}(M,\partial+\vartheta+\bar\theta), 0 \right).
\]

By the isomorphism $H^{\ast}(M,d+\vartheta-\bar\vartheta+\theta+\bar\theta)\cong H^{\ast}(M,\bar\partial-\bar\vartheta+\theta)$ and Lemma \ref{si}, for $t\in \R^{\ast}$, we obtain
\begin{multline*}
\dim H^{\ast}(M,d+\vartheta-\bar\vartheta+\theta+\bar\theta)= \dim H^{\ast}(M,\bar\partial-\bar\vartheta+\theta)\\
=\dim H^{\ast}(M,\bar\partial-\bar\vartheta+t\theta)=\dim H^{\ast}(M,d+\vartheta-\bar\vartheta+t\theta+t\bar\theta).
\end{multline*}
Hence  the $\mu_{\R^{\ast}}$-symmetry on cohomologies holds.

We prove the hyper-formality as in \cite{KP}.
 The direct sum
\[ \bigoplus_{E_{\vartheta-\bar\vartheta},\theta} \left({\rm ker}(\partial+\vartheta+\bar\theta), \bar\partial-\bar\vartheta+\theta\right)
\]
is a sub-differential  graded algebra of $\overline{A}^{\ast}(M)=\bigoplus_{E_{\vartheta-\bar\vartheta}\in F(M), \theta} (A^{\ast}(M),d+\vartheta-\bar\vartheta+\theta+\bar\theta)$.
We have the differential graded algebra quasi-isomorphisms
\[\bigoplus_{E_{\vartheta-\bar\vartheta},\theta} \left({\rm ker}(\partial+\vartheta+\bar\theta), \bar\partial-\bar\vartheta+\theta\right)\to \overline{A}^{\ast}(M)
\]
and 
\[\bigoplus_{E_{\vartheta-\bar\vartheta},\theta} \left({\rm ker}(\partial+\vartheta+\bar\theta), \bar\partial-\bar\vartheta+\theta\right)\to \bigoplus_{E_{\vartheta-\bar\vartheta} ,\theta} H^{\ast}(M,d+\vartheta-\bar\vartheta+\theta+\bar\theta).
\]
Hence $\overline{A}^{\ast}(M)$ is formal.
\end{proof}

\section{Algebraic hulls}\label{alg}

We review the notion of  algebraic hulls.
\begin{proposition}{\rm (\cite[Proposition 4.40]{R})}\label{ttt}
Let $G$ be a simply connected solvable Lie group (resp. torsion-free 
polycyclic group).
Then there exists a unique $\R$-algebraic group ${\bf H}_{G}$ with an injective group homomorphism $\psi :G\to {\bf H}_{G}(\R)  $ 
so that:
\begin{itemize}  \item $\psi (G)$ is Zariski-dense in $\bf{H}_{G}$.
\item    $Z_{{\bf H}_{G}}({\bf U}({\bf H}_{G}))\subset {\bf U}({\bf H}_{G})$ where ${\bf U}({\bf H}_{G})$ is the unipotent radical of ${\bf H}_{G}$.
\item $\dim {\bf U}({\bf H}_{G})$=${\rm dim}\,G$(resp. ${\rm rank}\, G$).   
\end{itemize}
Such ${\bf H}_{G}$ is called the algebraic hull of $G$.
\end{proposition}
We denote ${\bf U}_{G}={\bf U}({\bf H}_{G})$ and call ${\bf U}_{G}$ the unipotent hull of $G$.

In \cite{K} or \cite{K2}, the author showed the following proposition.
\begin{proposition}\label{abab}
Let $G$ be a simply connected solvable Lie group.
Then ${\bf U}_{G}$ is abelian if and only if  $G=\R^{n}\ltimes_{\varphi} \R^{m}$ such that the action $\varphi:\R^{n}\to {\rm  Aut} (\R^{m})$ is semi-simple.
\end{proposition} 

In this section,  we show the following proposition.
\begin{proposition}\label{betab}
Let $M$ be a solvmanifold and $\Gamma$ the fundamental group of $M$.
If ${\bf U}_{\Gamma}$ is abelian, $M$ is written as $G/\Gamma$ where $G=\R^{n}\ltimes_{\varphi} \R^{m} $ for $n=\dim H^{1}(M,\R)$ such that the action $\varphi:\R^{n}\to {\rm Aut}(\R^{m})$ is  semi-simple.
\end{proposition}
\begin{proof}
For a simply connected solvable Lie group $G$ which contains $\Gamma$ as a lattice.
Consider the algebraic hulls ${\bf H}_{G} $ and ${\bf H}_{\Gamma}$.
Then we have ${\bf H}_{\Gamma}\subset {\bf H}_{G} $ and ${\bf U}_{G}={\bf U}_{\Gamma}$ (see \cite[Proof of Theorem 4.34]{R}).
Let $N$ be a nilpotent subgroup in $ {\bf H}_{G} $.
Then  we have the direct product $N=N_{s}\times N_{u}$ where $N_{s}$ (resp. $N_{u}$) is the set of the semi-simple (resp. unipotent) part of the Jordan decomposition of the elements of $N$  (see \cite{Seg}).
Since $N_{s}$ is contained in an algebraic torus and $N_{u}\subset {\bf U}_{G}$ (see \cite[Theorem 10.6]{Bor}), 
 $N=N_{s}\times N_{u}$ is abelian.
Hence any nilpotent subgroup in $\Gamma$ or $G$ is abelian.

As in \cite{KM}, we have a normal nilpotent subgroup $\Delta$ in $\Gamma$ such that  $[\Gamma,\Gamma]$ is  a finite index subgroup in $\Delta$ and $\Gamma/\Delta$ is free abelian.
By the above argument, $\Delta$ is abelian.
Hence for $n=\dim H^{1}(M,\R)$, we have
\[\xymatrix{
	0\ar[r]&\Delta\ar[r]&\Gamma\ar[r] &\Z^{n} \ar[r] &0 .
	 }
\]
Since $\Delta$ is abelian,  as in \cite{Aus}, we have a simply connected solvable Lie group $G$ such that
\[\xymatrix{
	0\ar[r]&\Delta\otimes \R\ar[r]&G\ar[r]^{p} &\Z^{n}\otimes \R \ar[r] &0 .
	 }
\]
By using \cite[Theorem 2.2]{Dek}, there exists a simply connected nilpotent Lie subgroup $C$ such that $G=C\cdot (\Delta\otimes \R)$.
By the above argument, $C$ is abelian.
For  $p:C\to C/C\cap (\Delta\otimes \R)\cong G/(\Delta\otimes \R) \cong \Z^{n}\otimes \R$, we can take a homomorphism $q:C/C\cap (\Delta\otimes \R)\to C\subset G$ such that $p\circ q={\rm id}$ and hence we have a splitting $G=\Z^{n}\otimes \R\ltimes_{\varphi} \Delta\otimes \R$.
By Proposition \ref{abab}, $\varphi$ is semi-simple. Hence the proposition follows.
\end{proof}

\section{Cohomology of solvmanifolds}\label{coh}

Let $G$ be a simply connected   solvable Lie group with a lattice $\Gamma$ and $\g$ the Lie algebra of $G$.
Let $N$ be the nilradical (i.e. maximal connected nilpotent normal subgroup)  of $G$.
Then $N\cap \Gamma$ is a lattice in $N$ and $\Gamma/(\Gamma\cap N)$ is a lattice in $G/N$.
We consider the cochain complex $\bigwedge \g^{\ast}_{\C}$ with the  differential $d$ which is the dual to the Lie bracket of $\g$.

Let ${\mathcal C}(G,N)=\{\alpha\in {\rm Hom}(G,\C^{\ast}) \vert \alpha_{\vert_{N}}=1\}$
and  ${\mathcal C}(G,N,\Gamma)$ the set  of characters of $\Gamma$ given by the restrictions of  $\alpha\in {\mathcal C}(G,N)$.
Since $\Gamma/(\Gamma\cap N)$ is a lattice in $G/N$, the set  ${\mathcal C}(G,N,\Gamma)$ is identified with ${\rm Hom}(\Gamma/(\Gamma\cap N), \C^{\ast})$.
For $\alpha\in {\mathcal C}(G,N)$, for the $1$-dimensional vector space $V_{\alpha}$ with the $G$-action via $\alpha$,  we consider the cochain complex $\bigwedge {\frak g}^{\ast}_{\C}\otimes V_{\alpha}$
such that  for $x\otimes v_{\alpha}\in \bigwedge {\frak g}^{\ast}_{\C}\otimes V_{\alpha}$
the differential $d_{\alpha}$ on $\bigwedge {\frak g}^{\ast}_{\C}\otimes V_{\alpha}$ is given by
\[d_{\alpha}(x\otimes v_{\alpha})= dx\otimes v_{\alpha}+\alpha^{-1}d\alpha\wedge x\otimes v_{\alpha}.\]
We consider $\bigwedge {\frak g}^{\ast}_{\C}\otimes V_{\alpha}$ as the space of the $\alpha$-twisted left-invariant differential forms with values  in the flat bundle $E_{\alpha^{-1}d\alpha}$.
We have the inclusion
\[\left(\bigwedge {\frak g}^{\ast}_{\C}\otimes V_{\alpha}, d_{\alpha}\right)\subset \left( A^{\ast}_{\C}(G/\Gamma),d+\alpha^{-1}d\alpha\right).
\]
We consider the direct sum
\[\bigoplus_{\alpha\in {\mathcal C}(G,N)}\left(\bigwedge {\frak g}^{\ast}_{\C}\otimes V_{\alpha},d_{\alpha}\right).\]

Let $F(G/\Gamma,N)=\{E_{\alpha^{-1}d\alpha}\in F(G/\Gamma)\vert \alpha\in {\mathcal C}(G,N)\}$.
We notice that the map ${\mathcal C}(G,N)\ni \alpha\mapsto E_{\alpha^{-1}d\alpha}\in F(G/\Gamma,N)$ is not injective but this map gives the $1-1$ correspondense ${\mathcal C}(G,N,\Gamma)\to F(G/\Gamma,N)$.
We also consider the direct sum
\[\bigoplus_{E_{\phi}\in F(G/\Gamma,N)}\left(A^{\ast}(G/\Gamma),d+\phi\right).
\]
Then we have the inclusion
\[\bigoplus_{\alpha\in {\mathcal C}(G,N)}\left(\bigwedge {\frak g}^{\ast}_{\C}\otimes V_{\alpha},d_{\alpha}\right)\hookrightarrow \bigoplus_{E_{\phi}\in F(G/\Gamma,N)}\left(A^{\ast}(G/\Gamma),d+\phi\right).\]
\begin{theorem}{\rm (\cite[Theorem 1.3]{KDD})}
This  inclusion induces a cohomology isomorphism
\[\bigoplus_{\alpha\in {\mathcal C}(G,N)}H^{\ast}(\g, V_{\alpha})\cong\bigoplus_{ E_{\phi}\in F(G/\Gamma,N)} H^{\ast}(G/\Gamma, d+\phi).
\]
\end{theorem}

Take a simply connected nilpotent subgroup $C\subset G$ such that $G=C\cdot N$ as in \cite[Proposition 3.3]{dek}.
 Since $C$ is nilpotent, the map 
\[\Phi:C\ni c \mapsto \bigoplus_{\alpha\in {\mathcal C}(G,N)} ({\rm Ad}_{c})_{s}\otimes \alpha(c) \in {\rm Aut}\left(\bigoplus_{\alpha\in {\mathcal C}(G,N)}\bigwedge {\frak g}^{\ast}_{\C}\otimes V_{\alpha}\right)\]
is a homomorphism where $({\rm Ad}_{c})_{s}$ is the semi-simple part of the Jordan decomposition of $({\rm Ad}_{c})$.
We denote by 
\[\left(\bigoplus_{\alpha\in {\mathcal C}(G,N)}\bigwedge {\frak g}^{\ast}_{\C}\otimes V_{\alpha}\right)^{\Phi(C)}
\]
the subcomplex consisting of the $\Phi(C)$-invariant elements.
\begin{lemma}{\rm (\cite[Lemma 5.2]{KDD})}
The inclusion
\[\left(\bigoplus_{\alpha\in{\mathcal C}(G,N)}\bigwedge {\frak g}^{\ast}_{\C}\otimes V_{\alpha}\right)^{\Phi(C)}\subset \bigoplus_{\alpha\in{\mathcal C}(G,N)}\bigwedge {\frak g}^{\ast}_{\C}\otimes V_{\alpha}
\]
induces a cohomology isomorphism.
\end{lemma}
We have a basis $X_{1},\dots,X_{n}$ of $\g_{\C}$ such that $({\rm Ad}_{c})_{s}={\rm diag} (\alpha_{1}(c),\dots,\alpha_{n}(c))$ for $c\in C$.
Let $x_{1},\dots, x_{n}$ be the basis of $\g_{\C}^{\ast}$ which is dual to $X_{1},\dots ,X_{n}$.
Let $v_{\alpha}$ be a basis of $V_{\alpha}$ for each character $\alpha \in{\mathcal C}(G,N)$.
By $G=C\cdot N$, we have $G/N=C/C\cap N$ and hence we have ${\mathcal C}(G,N)=\{\alpha\in {\rm Hom}(C,\C^{\ast}) \vert \alpha_{\vert_{C\cap N}}=1\}$.
We have
\[
\left(\bigoplus_{\alpha}\bigwedge {\frak g}^{\ast}_{\C}\otimes V_{\alpha}\right)^{\Phi(C)}
=\bigwedge \langle x_{1}\otimes v_{\alpha_{1}},\dots, x_{n} \otimes v_{\alpha_{n}} \rangle .
\]

Hence obtain the following corollaries.
\begin{corollary}
The inclusion 
\[\bigwedge \langle x_{1}\otimes v_{\alpha_{1}},\dots, x_{n} \otimes v_{\alpha_{n}} \rangle\hookrightarrow \bigoplus_{E_{\phi}\in F(G/\Gamma,N)}\left(A^{\ast}(G/\Gamma),d+\phi\right)
\]
induces a cohomology isomorphism.
\end{corollary}

\begin{corollary}\label{cocoe}
For $E_{\phi}\in F(G/\Gamma,N)$, let 
\[A_{\phi}^{\ast}=  \left\langle x_{I}\otimes v_{\alpha_{I}}\in \bigwedge \langle x_{1}\otimes v_{\alpha_{1}},\dots, x_{n} \otimes v_{\alpha_{n}} \rangle  \vert I\subset \{1,\dots, n\},\,\, E_{\alpha^{-1}_{I}d\alpha_{I}}=E_{\phi}\right\rangle
\]
where for a multi-index $I=\{i_{1},\dots,i_{p}\}\subset \{1,\dots, n\}$ we write $x_{I}=x_{i_{1}}\wedge \dots \wedge x_{i_{p}}$ and $\alpha_{I}=\alpha_{i_{1}}\cdots \alpha_{i_{p}}$.
Then the inclusion $A_{\phi}^{\ast} \subset \left( A^{\ast}_{\C}(M),d+\phi\right)$ induces a cohomology isomorphism.
\end{corollary}

\section{${\mathcal J}(M) $ on solvmanifolds}\label{sesollo}

Let $M$ be a compact manifold.
We consider the sets ${\mathcal J}^{p}(M)=\{ E_{\phi}\in F(M)\vert  H^{p}(M,d+\phi)\not=0\}$ and  ${\mathcal J}(M)=\bigcup{\mathcal J}^{p}(M)$ as in the Introduction.

If  $M$  is an aspherical manifold with $\pi_{1}(M)\cong \Gamma$, then we can identify  ${\mathcal J}^{p}(M)$ with the set   ${\mathcal J}^{p}(\Gamma)=\{\alpha\in {\mathcal C}(\Gamma)\vert H^{p}(\Gamma,V_\alpha)\not=0\}$ where $H^{\ast}(\Gamma,V_\alpha)$ is the group cohomology with values in the module $V_{\alpha}$ given by $\alpha$.
We also consider the set  ${\mathcal {\breve J}}^{p}(\Gamma)=\{\alpha\in {\rm Hom}(\Gamma,\C^{\ast})\vert H^{p}(\Gamma,V_\alpha)\not=0\}$ which is identified with the set  ${\mathcal {\breve J}}^{p}(M)$ on a aspherical manifold as in Remark \ref{remjul}.

\begin{lemma}[\cite{MPa}]\label{nillc}
Let $\Gamma$ be a torsion-free finitely generated nilpotent group.
Then we have ${\mathcal {\breve J}}(\Gamma)=\bigcup_{p}{\mathcal {\breve J}}^{p}(\Gamma)=\{1_{\Gamma}\}$ where $1_{\Gamma}$ is the trivial  character of $\Gamma$.
\end{lemma}

Let $G$ be a simply connected  solvable Lie group with a lattice $\Gamma$ and $\g$ the Lie algebra of $G$.
Then the solvmanifold $G/\Gamma$ is an aspherical manifold with $\pi_{1}(G/\Gamma)\cong \Gamma$.

\begin{corollary}\label{soto}
For $E_{\phi}\in{\mathcal {\breve J}}^{p}(G/\Gamma)$, we have $E_{\phi}\in F(G/\Gamma,N)$.
Hence, we have  
\[{\mathcal {\breve J}}^{p}(G/\Gamma)={\mathcal J}^{p}(G/\Gamma)\subset F(G/\Gamma,N).\] 
\end{corollary}
\begin{proof}
Consider the character $\alpha\in{\mathcal {\breve J}}^{p}(\Gamma) $ which corresponds to $E_{\phi}\in {\mathcal {\breve J}}^{p}(G/\Gamma)$.
If the restriction $\alpha_{\vert \Gamma\cap N}$ is non-trivial,
then  considering the Hochschild-Serre spectral sequence $E_{\ast}^{\ast,\ast}$, by Lemma \ref{nillc} we have
\[E_{2}^{p,q}=H^{p}\left(\Gamma/\Gamma\cap N, H^{q}(\Gamma\cap N,V_{\alpha})\right)=0
\]
and hence  $H^{\ast}(G/\Gamma, E_{\phi})=H^{\ast}(\Gamma,V_{\alpha})=0$ since  the Hochschild-Serre spectral sequence converges to $H^{\ast}(\Gamma, V_{\alpha})$.
Hence the restriction $\alpha_{\vert \Gamma\cap N}$ is trivial and $\alpha$ induces a character on $\Gamma/\Gamma\cap N$.
Since   $\Gamma/\Gamma\cap N$ is a lattice in the Abelian Lie group $G/N$, we can extend $\alpha$ to a character of $G$ whose restriction on $N$ is trivial.
Hence we can say  $E_{\phi}\in F(G/\Gamma,N)$.
\end{proof}

%\begin{corollary}
%We have a quasi-isomorphism
%\[\bigwedge {\frak u}^{\ast}\to \overline{A}^{\ast}(G/\Gamma).
%\]
%\end{corollary}
Now we use the same setting as in Section \ref{coh}.
By Corollary \ref{soto}, the inclusion
\[\bigoplus_{E_{\phi}\in F(G/\Gamma,N)}\left(A^{\ast}(G/\Gamma),d+\phi\right)\subset \overline {A}^{\ast}(G/\Gamma)=\bigoplus_{E_{\phi}\in F(G/\Gamma)}\left(A^{\ast}_{\C}(G/\Gamma), d+\phi\right) 
\]
induces a cohomology isomorphism and hence 
considering the inclusions
\[\bigwedge \langle x_{1}\otimes v_{\alpha_{1}},\dots, x_{n} \otimes v_{\alpha_{n}} \rangle \subset \bigoplus_{E_{\phi}\in F(G/\Gamma,N)}\left(A^{\ast}(G/\Gamma),d+\phi\right)\subset  \overline {A}^{\ast}(G/\Gamma),
\]
the following theorem holds.
\begin{theorem}\label{isomo}
The inclusion
\[\bigwedge \langle x_{1}\otimes v_{\alpha_{1}},\dots, x_{n} \otimes v_{\alpha_{n}} \rangle \hookrightarrow \overline {A}^{\ast}(G/\Gamma)\]
induces a cohomology isomorphism.
\end{theorem}

We  denote by ${\mathcal K}^{p}(G)$ the set of the characters of $G$ which are written as $\alpha_{I}$ for some multi-index $I=\{i_{1},\dots,i_{p}\}\subset \{1,\dots, n\}$.
Let  ${\mathcal K}^{p}(G,\Gamma)=\{E_{\alpha^{-1}d\alpha}\in F(G/\Gamma)\vert \alpha\in{\mathcal K}^{p}(G)\}$.
Then the following proposition follows from Corollary \ref{soto} and Corollary \ref{cocoe}.

\begin{proposition}\label{jtii}
We have $ {\mathcal J}^{p}(G/\Gamma) \subset{\mathcal K}^{p}(G,\Gamma)$.
\end{proposition}

We have a differential graded algebra isomorphism
\[\bigwedge \langle x_{1}\otimes v_{\alpha_{1}},\dots, x_{n} \otimes v_{\alpha_{n}} \rangle\cong \bigwedge {\frak u}^{\ast}\]
where $\frak u$ is the Lie algebra of the unipotent hull ${\bf U}_{G}$ of $G$ as in Section \ref{alg} (see \cite[Remark 4]{KDD} and \cite{K2}).
By Theorem \ref{isomo}, we obtain the following corollary.
\begin{corollary}
We have a quasi-isomorphism
\[\bigwedge {\frak u}^{\ast}\to \overline {A}^{\ast}(G/\Gamma)
\]
and hence $\bigwedge {\frak u}^{\ast}$ is the Sullivan minimal model of $\overline {A}^{\ast}(G/\Gamma)$.
\end{corollary}

In \cite{H}, Hasegawa proved that
for  a nilpotent Lie algebra $\frak n$,
the differential graded algebra $\bigwedge {\frak n}^{\ast}$ is formal if and only if $\frak n$ is abelian.
Hence we obtain the following corollary.

\begin{corollary}\label{formsol}
A solvmanifold $G/\Gamma$ is hyper-formal if and only if the unipotent hull ${\bf U}_{G}$ is abelian.
\end{corollary}

Suppose we have $[\omega]\in H^{2}( {\frak n})$ such that $[\omega]^{n}\not=0$ where $2n=\dim \frak n$.
In \cite{BG} (see also \cite[Section 4.6.4]{Fe}), Benson and Gordon proved that
 for any $0\leq i\leq n$ the linear operator
\[[\omega]^{n-i}\wedge: H^{i}( {\frak n})\to H^{2n-i}( {\frak n})
\]
is an isomorphism if and only if $\frak n$ is abelian.
Hence we obtain the following corollary.

\begin{corollary}\label{lefsol}
Suppose a solvmanifold $(G/\Gamma,\omega)$ admits a symplectic structure $\omega$.
Then $(G/\Gamma,\omega)$ is hyper-hard-Lefschetz if and only if the unipotent hull ${\bf U}_{G}$ is Abelian.
\end{corollary}

Take a simply connected nilpotent subgroup $C\subset G$ such that $G=C\cdot N$.
We define the action  of $C$ on the differential graded algebra $\bigwedge{\frak u}^{\ast}\cong \bigwedge\langle x_{1}\otimes v_{\alpha_{1}},\dots, x_{n}\otimes v_{\alpha_{n}}\rangle$ as
\[c\cdot (x_{i}\otimes v_{\alpha_{i}})=\alpha_{i}x_{i}\otimes v_{\alpha_{i}}.
\]
Consider the induced action of $C$ on the cohomology $H^{\ast}(\frak u)$.
We take the weight decomposition $H^{p}(\frak u)=\bigoplus V^p_{\mu_{j}}$ of this $C$-action.

\begin{theorem}\label{loci}
Let ${\mathcal L}^{p}(G,\Gamma)=\{E_{\mu_{j}^{-1}d\mu_{j}}\in F(G/\Gamma) \vert V^p_{\mu_{j}}\not=0 \}$. 
Then we have  ${\mathcal J}^{p}(M)={\mathcal L}^{p}(G,\Gamma)$.
\end{theorem}
\begin{proof}
Consider the $C$-action
\[\Phi:C\ni c \mapsto ({\rm Ad}_{c})_{s}\in {\rm Aut}(\bigwedge \g_{\C}^{\ast})\]
on $\bigwedge \g_{\C}^{\ast}$ and the weight decomposition
\[\bigwedge \g_{\C}^{\ast}=\bigoplus W^{\ast}_{\nu_{i}}\]
of this action.
Then we have 
\[ \bigwedge {\frak u}^{\ast}\cong \bigwedge \langle x_{1}\otimes v_{\alpha_{1}},\dots, x_{n} \otimes v_{\alpha_{n}} \rangle\cong \bigoplus W^{\ast}_{\nu_{i}}\otimes \langle v_{\nu_{i}^{-1}}\rangle,
\]

\[H^{p}(W^{\ast}_{\nu_{i}}\otimes \langle v_{\nu_{i}^{-1}}\rangle)=V^{p}_{\nu^{-1}_{j}}
\]
and
\[
A_{\phi}^{\ast}=\bigoplus_{E_{\nu_{i}^{-1}d\nu_{i}}=E_{\phi}}W^{\ast}_{\nu_{i}}\otimes \langle v_{\nu_{i}^{-1}}\rangle\]
where $A_{\phi}^{\ast}$ is defined in 
 Theorem \ref{isomo}.
By  Theorem \ref{isomo}, we obtain
\[H^{p}(G/\Gamma,d+\phi)\cong H^{p}(A_{\phi}^{\ast})=\bigoplus_{E_{\nu_{i}d\nu^{-1}_{i}}=E_{\phi}}H^{p}(W^{\ast}_{\nu_{i}}\otimes \langle v_{\nu_{i}^{-1}}\rangle)=\bigoplus_{E_{\nu_{i}d\nu^{-1}_{i}}=E_{\phi}}V^{p}_{\nu^{-1}_{i}}.
\]
Hence $H^{p}(G/\Gamma,d+\phi)\not=0$ if and only if  $V^{p}_{\mu_{j}}\not=0$  such that $E_{\mu_{j}^{-1}d\mu_{j}}=E_{\phi}$.
\end{proof}

\begin{remark}\label{remabci}
In general,  by Proposition \ref{jtii}, we have ${\mathcal L}^{p}(G,\Gamma)\subset{\mathcal K}^{p}(G,\Gamma)$ but ${\mathcal L}^{p}(G,\Gamma)\not={\mathcal K}^{p}(G,\Gamma)$.
If ${\bf U}_{G}$ is Abelian, then we have $H^{\ast}({\frak u})=\bigwedge^{\ast}{\frak u}$ and hence  ${\mathcal L}^{p}(G,\Gamma)={\mathcal K}^{p}(G,\Gamma)$.
\end{remark}

\begin{corollary}\label{finnj}
${\mathcal J}(G/\Gamma)$ is a finite set.
\end{corollary}

\begin{proposition}\label{sic}
Suppose that one of the characters $\alpha_{1},\dots,\alpha_{n}$ is non-unitary.
Then for a solvmanifold $G/\Gamma$, there exists a non-unitary flat bundle $E_{\phi}\in F(M)$ such that 
\[H^{1}(G/\Gamma,d+\phi)\not=0.
\]
\end{proposition}
\begin{proof}
By Theorem \ref{loci}, it is sufficient to prove that there exists a non-unitary  weight $\mu_{j}$ of the above action of $C$ on the cohomology $H^{1}(\frak u)$ such that $V^p_{\mu_{j}}\not=0$.
This  follows from the following lemma.
\end{proof}
\begin{lemma}
Let $\frak u$ be   a  $\C$-nilpotent Lie algebra.
Let $A \in{\rm Aut}(\frak u)$ be a semi-simple automorphism.
Consider the automorphism $\wedge A^{\ast}$ on $\bigwedge {\frak u}^{\ast}$ induced by $A$ and its restriction on $H^{1}({\frak u})={\rm ker}(d_{\vert \bigwedge^1 {\frak u}^{\ast}})$.
If there exists a non-unitary eigenvalue of $A$, then
the action on $H^{1}({\frak u})={\rm ker}(d_{\vert \bigwedge^1 {\frak u}^{\ast}})$ induced by $A$ has  a non-unitary eigenvalue.
\end{lemma}

\begin{proof}
We will show  inductively on the  nilpotency $s$ of $\frak u$ (i.e. $s$ is the number such that $C^{s-1}{\frak u}\not=0$ and $C^{s}{\frak u}=0$ where $C^{i}{\frak u}$ is the $i$-th term of  the lower central series of ${\frak u}$.)

In case $s=1$, since $H^{1}({\frak u})={\rm ker}(d_{\vert \bigwedge^1 {\frak u}^{\ast}})=\frak u^{\ast}$, there is nothing to prove.

We assume that the statement holds in the case $s\le n$ and $\frak u$ has  nilpotency $n+1$.
Then we have the decomposition
\[\bigwedge {\frak u}^{\ast}=\bigwedge \left( \left({\frak u}/C^{n}{\frak u}\right)^{\ast}\right)\otimes \bigwedge\left(C^{n}{\frak u}\right)^{\ast}
\]
such that $d\left(C^{n}{\frak u}\right)^{\ast}\subset \bigwedge^2 \left( {\frak u}/C^{n}{\frak u}\right)^{\ast}$.
Since $A \left(C^{n}{\frak u}\right)\subset \left(C^{n}{\frak u}\right)$, we have 
$A^{\ast}\left(\left(C^{n}{\frak u}\right)^{\ast}\right)\subset \left(C^{n}{\frak u}\right)^{\ast}$ and  $A^{\ast}\left(\left( {\frak u}/C^{n}{\frak u}\right)^{\ast}\right) \subset \left( {\frak u}/C^{n}{\frak u}\right)^{\ast}$. 
By the assumption of $A$, there exists a non-unitary eigenvalue of $A^{\ast}$ on $\left(C^{n}{\frak u}\right)^{\ast}$ or  $\left( {\frak u}/C^{n}{\frak u}\right)^{\ast}$.
If there exists a non-unitary eigenvalue of $A^{\ast}$ on $\left( {\frak u}/C^{n}{\frak u}\right)^{\ast}$, then by the induction assumption, we can prove the statement.
If  there exists a non-unitary eigenvalue of $A^{\ast}$ on
$\left(C^{n}{\frak u}\right)^{\ast}$, then by $d\left(C^{n}{\frak u}\right)^{\ast}\subset \bigwedge^2 \left( {\frak u}/C^{n}{\frak u}\right)^{\ast}$, we have a non-unitary eigenvalue of $\wedge^{2}A^{\ast}$ on $\bigwedge^2 \left( {\frak u}/C^{n}{\frak u}\right)^{\ast}$ and hence we have a non-unitary eigenvalue of $A^{\ast}$ on $\left( {\frak u}/C^{n}{\frak u}\right)^{\ast}$.
Hence the lemma follows.
\end{proof}

\begin{theorem}\label{symunit}
Let $G/\Gamma$ be a solvmanifold.
$G/\Gamma$  has the $\mu_{\R^{\ast}}$-symmetry on cohomologies if and only if the characters $\alpha_{1},\dots,\alpha_{n}$ are all unitary characters.
\end{theorem}
\begin{proof}
Suppose that $G/\Gamma$  has the $\mu_{\R^{\ast}}$-symmetry on cohomologies.
By Proposition \ref{sic}, if  one of the characters $\alpha_{1},\dots,\alpha_{n}$ is non-unitary,
then by Proposition \ref{sic} there exists a non-unitary flat bundle $E_{\phi}\in F(M)$ such that 
\[H^{1}(G/\Gamma,d+\phi)\not=0.
\]
We have $\{\mu_{t}(E_{\phi})=E_{t{\rm  Re}\phi+\sqrt{-1}{\rm Im} \phi}\vert t\in \R^{\ast}\}\subset {\mathcal J}(G/\Gamma)$
and hence ${\mathcal J}(G/\Gamma)$ is a infinite set.
But since ${\mathcal J}(G/\Gamma)$ is a finite set by Corollary \ref{finnj}, the characters $\alpha_{1},\dots,\alpha_{n}$ are all unitary characters.

Suppose that the characters $\alpha_{1},\dots,\alpha_{n}$ are all unitary characters.
Then, by Proposition \ref{jtii}, ${\mathcal J}(G/\Gamma)$ consists of unitary flat bundles.
Since any unitary flat bundle is fixed by the $\R^{\ast}
$-action via $\mu$, $G/\Gamma$  has the $\mu_{\R^{\ast}}$-symmetry on cohomologies.
\end{proof}

\section{Proof of Theorem \ref{mainte}}
\begin{proof}
(1)$\Rightarrow$(2) follows from Proposition \ref{Hysym}.

We prove (2)$\Rightarrow$(4).
Write $M=G/\Gamma$.
By Corollary \ref{formsol}, the unipotent hull ${\bf U}_{G}={\bf U}_{\Gamma}$ is Abelian.
By  Proposition \ref{betab}, we can write $G=\R^{2k}\ltimes_{\varphi} \R^{2l} $ such that the action $\varphi:\R^{2k}\to {\rm Aut}(\R^{2l})$ is  semi-simple.
By using Theorem \ref{symunit}, we can easily check that all eigencharacters of $\phi$ are unitary.
Hence (2)$\Rightarrow$(4) follows.

Noting that the dimensions of the cohomologies of odd degree of hard Lefschetz symplectic manifolds are even (see \cite{bock}), 
as similar to the proof of (2)$\Rightarrow$(4), we can prove (3)$\Rightarrow$(4) by Corollary \ref{lefsol}.

We prove (4)$\Rightarrow$(5).
By the assumption, we can take a $\varphi(\R^{2k})$-invariant flat K\"ahler metric $h$ on $\R^{2l} $.
For a flat K\"ahler metric $g$ on $\R^{2k}$, we have the left-invariant   K\"ahler metric $g\times h$ on $G$ and this induces a K\"ahler metric on $G/\Gamma$.
Hence (4)$\Rightarrow$(5) follows.

(5)$\Rightarrow$(1) and (5)$\Rightarrow$(3)  follow from Theorem \ref{simps}.

Now we have (1)$\Rightarrow$(2)$\Rightarrow$(4)$\Rightarrow$(5)$\Rightarrow$(1)
and (3)$\Rightarrow$(4)$\Rightarrow$(5)$\Rightarrow$(3).
Hence the theorem follows.
\end{proof}

\section{Computations of the cohomology $H^{\ast}(M,\bar\partial+\phi^{0,1}+\theta)$ on certain solvmanifolds}

Let $\frak g$ be a  Lie algebra with a complex structure $J$.
We consider the differential bi-graded algebra $\bigwedge^{\ast,\ast}\g_{\C}$ with the differential $\bar\partial$.
Let $\theta \in H^{1,0}(\bigwedge^{\ast,\ast}\g_{\C},\bar\partial)={\rm ker}\, \bar\partial_{\vert_{\bigwedge^{1,0}\g_{\C}}}$ such that $\theta\not=0$.
We consider the cohomology $H^{\ast}(\bigwedge\g_{\C},\bar\partial+\theta)$ which is the total cohomology of the double complex $(\bigwedge^{\ast,\ast}\g_{\C},\theta, \bar\partial)$.
Then we have the spectral sequence $E_{\ast}^{\ast,\ast}$ of  the double complex $(\bigwedge^{\ast,\ast}\g_{\C},\theta, \bar\partial)$ such that the first term $E_{1}^{\ast,\ast}$ is the cohomology of  $(\bigwedge^{\ast,\ast}\g_{\C},\theta)$ and $E_{\ast}^{\ast,\ast}$ converges to $H^{\ast}(\bigwedge\g_{\C},\bar\partial+\theta)$.
By simple computations, we obtain $E_{1}^{\ast,\ast}=0$ and hence  $H^{\ast}(\bigwedge\g_{\C},\bar\partial+\theta)=0$.

\begin{theorem}{\rm (\cite[Theorem 1]{sakane}, \cite[Main Theorem]{cordero-fernandez-gray-ugarte}, \cite[Theorem 2, Remark 4]{console-fino}, \cite[Theorem 1.10]{rollenske}, \cite[Corollary 3.10]{rollenske-survey}, \cite[Theorem 3.8]{angella-1})}
Let $G$ be a simply connected nilpotent Lie group with a lattice $\Gamma$ and left-invariant complex structure $J$.
Let $\g$ be the Lie algebra of $G$.
Suppose that one of the following conditions holds:
  \begin{itemize}
    \item $(G,J)$ is a complex Lie group;
    \item $J$ is an abelian complex structure;
    \item $J$ is a nilpotent complex structure;
    \item $J$ is a rational complex structure;
    \item $\mathfrak{g}$ admits a torus-bundle series compatible with $J$ and with the rational structure induced by $\Gamma$.
\end{itemize}

Then  the inclusion
\[\left(\bigwedge\nolimits^{\ast,\ast}\g^{\ast}_{\C},\partial,\bar\partial\right)\hookrightarrow \left(A^{\ast,\ast}(G/\Gamma),\partial,\bar\partial\right),
\]
induces an isomorphism
\[H^{\ast,\ast}(\bigwedge\nolimits^{\ast,\ast}\g^{\ast}_{\C},\bar\partial)\cong H^{\ast,\ast}(G/\Gamma, \bar\partial).\]
\end{theorem}

Let $(G,\Gamma, J)$ as in assumption of this theorem.
For a non-trivial holomorphic $1$-form $\theta \in H^{1,0}(G/\Gamma,\bar\partial)$, we have $\theta\in \bigwedge^{1,0}\g_{\C} $ and the injection
\[\left(\bigwedge \nolimits^{\ast,\ast}\g^{\ast}_{\C},\theta,\bar\partial\right)\hookrightarrow  \left(A^{\ast,\ast}(G/\Gamma), \theta, \bar\partial\right)
\]
 between double complexes, which induces a cohomology isomorphism
\[H^{\ast,\ast}(\bigwedge\nolimits^{\ast,\ast}\g_{\C}^{\ast},\bar\partial)\cong H^{\ast,\ast}(G/\Gamma, \bar\partial).\]
By using the spectral sequences of the double complexes, we can easily check  the isomorphism 
\[H^{\ast,\ast}(\bigwedge\nolimits^{\ast,\ast}\g_{\C}^{\ast},\bar\partial)\cong H^{\ast,\ast}(G/\Gamma, \bar\partial).\]
(see \cite[Proposition 1.1]{AK}).
By the above argument, we have
\[H^{\ast}(\bigwedge\g_{\C}^{\ast},\bar\partial+\theta)=0\]
and hence we obtain the following result.
\begin{corollary}
For any non-trivial holomorphic $1$-form $\theta \in H^{1,0}(G/\Gamma,\bar\partial)$, we have
\[
H^{\ast}(G/\Gamma, \bar\partial+\theta)=0.\]
\end{corollary}

Next we consider a solvable Lie group $G$ with the following assumption.
\begin{assumption}\label{Ass}
$G$ is the semi-direct product $\C^{n}\ltimes _{\varphi}N$ so that:

Let $\frak a$ and $\n$ be the Lie algebras of $\C^{n}$ and $N$ respectively.
\begin{enumerate}
\item  $N$ is a simply connected nilpotent Lie group with a left-invariant complex structure $J$.

\item For any $t\in \C^{n}$, $\phi(t)$ is a holomorphic automorphism of $(N,J)$.
\item $\varphi$ induces a semi-simple action on the Lie algebra $\n$ of $N$.
\item $G$ has a lattice $\Gamma$. (Then $\Gamma$ can be written by $\Gamma=\Gamma^{\prime}\ltimes_{\varphi}\Gamma^{\prime\prime}$ such that $\Gamma^{\prime}$ and $\Gamma^{\prime\prime}$ are  lattices of $\C^{n}$ and $N$ respectively, and for any $t\in \Gamma^{\prime}$ the action $\varphi(t)$ preserves $\Gamma^{\prime\prime}$.) 
\item The inclusion $\bigwedge^{\ast,\ast}\n^{\ast}_{\C}\subset A^{\ast,\ast}(N/\Gamma^{\prime\prime})$ induces an isomorphism 
\[H^{\ast,\ast}(\bigwedge\nolimits^{\ast,\ast}\n_{\C},\bar\partial)\cong H^{\ast,\ast}(N/\Gamma^{\prime\prime}, \bar\partial).\]
\end{enumerate}
\end{assumption}
Consider the set
${\mathcal C}(G,N)=\{\alpha\in {\rm Hom}(G,\C^{\ast}) \vert \alpha_{\vert_{N}}=1\}$.
Let $HL(G/N)$ be the set of  isomorphism classes of holomorphic $\C$-line bundles.
We define the subset $HL(G,N,\Gamma)=\{[E_{\alpha^{-1}d\alpha}]_{hol}\in HL(G/N)\vert\alpha\in {\mathcal C}(G,N)\}$
where $[E_{\alpha^{-1}d\alpha}]_{hol}$ is the holomorphically isomorphism class containing a flat bundle $E_{\alpha^{-1}d\alpha}$.
We consider the bi-graded cochain complex $\left(A^{\ast,\ast}(G/\Gamma), \bar\partial+\phi^{0,1}\right)$ as  the Dolbeault complex with values in a holomorphic flat bundle $E_{\phi}$.
We consider the direct sum
\[\bigoplus_{[E_{\phi}]_{hol}\in HL(G,N,\Gamma)}\left(A^{\ast,\ast}(G/\Gamma), \bar\partial+\phi^{0,1}\right).
\]
Then by the wedge products and the tensor products, this direct sum  is a differential bi-graded algebra.

\begin{theorem}[\cite{Kd}]\label{MMMTTT}
There exists a differential bi-graded sub-algebra $A^{\ast,\ast}$ of 
\[\bigoplus_{[E_{\phi}]_{hol}\in HL(G,N,\Gamma)}\left(A^{\ast,\ast}(G/\Gamma), \bar\partial+\phi^{0,1}\right).
\]
 such that
we have a  differential bi-graded algebra isomorphism $\iota : \bigwedge ^{\ast,\ast}({\frak a}\oplus \n)^{\ast}_{\C}\cong A^{\ast,\ast}$ and
 the inclusion 
\[A^{\ast,\ast}\hookrightarrow \bigoplus_{[E_{\phi}]_{hol}\in HL(G,N,\Gamma)}\left(A^{\ast,\ast}(G/\Gamma), \bar\partial+\phi^{0,1}\right)\] induces a cohomology isomorphism. 
\end{theorem}

For a non-trivial holomorphic $1$-form $\theta \in H^{1,0}(G/\Gamma,\bar\partial)$, by Theorem \ref{MMMTTT}, we have $\theta\in A^{1,0} $ and by the differential bi-graded algebra isomorphism $\iota : \bigwedge ^{\ast,\ast}({\frak a}\oplus \n)^{\ast}_{\C}\cong A^{\ast,\ast}$, we have the injection
\[\left(\bigwedge  \nolimits^{\ast,\ast}({\frak a}\oplus \n)^{\ast}_{\C},\theta,\bar\partial\right)\hookrightarrow  \bigoplus_{[E_{\phi}]_{hol}\in HL(G,N,\Gamma)}\left(A^{\ast,\ast}(G/\Gamma), \theta, \bar\partial+\phi^{0,1}\right)
\]
 between double complexes, which induces a cohomology isomorphism
\[H^{\ast,\ast}\left(\bigwedge \nolimits^{\ast,\ast}({\frak a}\oplus \n)^{\ast},\bar\partial\right)\cong \bigoplus_{[E_{\phi}]_{hol}\in HL(G,N,\Gamma)}H^{\ast,\ast}(G/\Gamma,  \bar\partial+\phi^{0,1}).
\]
By using the spectral sequences of the double complexes, we can easily check  the isomorphism 
\[H^{\ast}(\bigwedge({\frak a}\oplus \n)^{\ast}_{\C},\bar\partial+\theta)\cong \bigoplus_{[E_{\phi}]_{hol}\in HL(G,N,\Gamma)}H^{\ast,\ast}(G/\Gamma,  \bar\partial+\phi^{0,1}+\theta)
\]
(see \cite[Proposition 1.1]{AK}).
By the above argument, we have
\[H^{\ast}(\bigwedge ({\frak a}\oplus \n)^{\ast}_{\C},\bar\partial+\theta)=0\]
and 
\[\bigoplus_{[E_{\phi}]_{hol}\in HL(G,N,\Gamma)}H^{\ast,\ast}(G/\Gamma,  \bar\partial+\phi^{0,1}+\theta)=0.
\]
Hence we obtain the following cohomology vanishing result.
\begin{corollary}\label{holodol}
For a flat holomorphic bundle  $[E_{\phi}]_{hol}\in HL(G,N,\Gamma)$ and a non-trivial holomorphic $1$-form $\theta \in H^{1,0}(G/\Gamma,\bar\partial)$,
we have 
\[H^{\ast}(G/\Gamma,  \bar\partial+\phi^{0,1}+\theta)=0.
\]
\end{corollary}

%By these conputations, we suggest the following conjecture and problem.
%\begin{conjecture}
%For a solvmanifold $G/\Gamma$ with a left-invariant complex structure, for any non-trivial holomorphic $1$-form $\theta \in H^{1,0}(G/\Gamma,\bar\partial)$, we have
%\[H^{\ast,\ast}(G/\Gamma, \bar\partial+\theta)=0.\]
%\end{conjecture}

%\begin{problem}
%Provide examples of complex manifolds $M$ such that 
%for a non-trivial holomorphic $1$-form $\theta \in H^{1,0}(M,\bar\partial)$ we have
%\[H^{\ast,\ast}(M, \bar\partial+\theta)\not=0.\]
%\end{problem}

\section{Examples}\label{secex}
Let $G=\C\ltimes _{\phi}\C^{2}$ such that $\phi(z_{1})=\left(
\begin{array}{cc}
e^{\frac{z_{1}+{\bar z_{1}}}{2}}& 0  \\
0&    e^{-\frac{z+{\bar z_{1}}}{2}}  
\end{array}
\right)$.
Then for some $a\in \R$  the matrix $\left(
\begin{array}{cc}
e^{a}& 0  \\
0&    e^{-a}  
\end{array}
\right)$
 is conjugate to an element of $SL(2,\Z)$.
 Hence, for any $0\not=b\in \R$, we have a lattice $\Gamma=(a\Z+b\sqrt{-1}\Z )\ltimes \Gamma^{\prime\prime}$ such that $\Gamma^{\prime\prime} $ is a lattice of $\C^{2}$.
 Then for a coordinate $(z_{1},z_{2},z_{3})\in \C\ltimes _{\phi}\C^{2}$, for the Lie algebra $\g$ of $G$,  we have
\[\bigwedge \nolimits^{\ast,\ast}\g^{\ast}_{\C}
=\bigwedge \nolimits^{\ast,\ast}\langle dz_{1},\, e^{-\frac{z_{1}+{\bar z_{1}}}{2}}dz_{2}, e^{\frac{z_{1}+{\bar z_{1}}}{2}}dz_{3}\rangle\otimes \langle d\bar z_{1},\, e^{-\frac{z_{1}+{\bar z_{1}}}{2}}d\bar z_{2},\, e^{\frac{z_{1}+{\bar z_{1}}}{2}}d\bar z_{3}\rangle.
\]
We have a left-invariant symplectic structure \[\omega=\sqrt{-1}dz_{1}\wedge d\bar z_{1}+dz_{2}\wedge d\bar z_{3}+d\bar z_{2}\wedge d z_{3}.\]
Take $b\not\in \pi\Z$, then $G/\Gamma$ satisfies the strong-Hodge-decomposition (see \cite{AK}).
Moreover by Corollary \ref{formsol} and Corollary \ref{lefsol},
$G/\Gamma$ is hyper-formal and  hyper-hard-Lefschetz.
On the other hand, by Theorem \ref{mainte}, $G/\Gamma$ does not satisfy  the hyper-strong-Hodge-decomposition.

We observe in more detail. 
By Remark \ref{remabci}, we have 
\[{\mathcal J}(G/\Gamma)=\left\{E_{0}, E_{\frac{dz_{1}+{d\bar z_{1}}}{2}},E_{-\frac{dz_{1}+{d\bar z_{1}}}{2}}, E_{dz_{1}+{d\bar z_{1}}},E_{-dz_{1}-{d\bar z_{1}}}\right\}
\]
Hence we have \[H^{\ast}\left(G/\Gamma,d+\frac{dz_{1}+{d\bar z_{1}}}{2}\right)\not=0.\]
But by Corollary \ref{holodol}, we have 
\[H^{\ast}\left(G/\Gamma,\bar\partial+\frac{dz_{1}}{2}\right)=0\]
 and hence 
\[H^{\ast}\left(G/\Gamma,d+\frac{dz_{1}+{d\bar z_{1}}}{2}\right)\not\cong H^{\ast}\left(G/\Gamma,\bar\partial+\frac{dz_{1}}{2}\right).\]
This implies that $G/\Gamma$ does not satisfy  the hyper-strong-Hodge-decomposition (see the proof of Proposition \ref{Hysym}).
In particular,  $G/\Gamma$ does not  admit a K\"ahler structure.

\end{document}